\documentclass[12pt]{amsart}
\usepackage{amssymb,verbatim}
\usepackage{graphicx}
\usepackage{pstricks} %% !! Incompatible with pdflatex...
\psset{unit=1cm}
\psset{arrowsize=3pt 2}

\newtheorem{theorem}{Theorem}[section]

\newtheorem{lemma}[theorem]{Lemma}

\theoremstyle{definition}

\newcommand{\0}{\emptyset}

\newcommand{\ol}{\overline}

\newcommand{\tS}{\tilde S}

\newcommand{\e}{\varepsilon}

\newcommand{\al}{\alpha}

\newcommand{\be}{\beta}
\newcommand{\ga}{\gamma}

\newcommand{\si}{\sigma}
\newcommand{\ta}{\theta}
\newcommand{\om}{\omega}
\newcommand{\Da}{\Delta}

\newcommand{\nin}{\not\in}
\newcommand{\imp}{\mbox{Imp}}

\newcommand{\dist}{\mbox{dist}}
\newcommand{\diam}{\mbox{diam}}

\newcommand{\iy}{\infty}
\newcommand{\C}{\mbox{$\mathbb{C}$}}

\newcommand{\D}{\mbox{$\mathbb{D}$}}

\begin{document}

\date{June 11, 2007; the revised version October 14, 2008}
\title[The solar Julia sets]
{The solar Julia sets of basic quadratic Cremer polynomials}

\author{A.~Blokh}

\thanks{The first author was partially
supported by NSF grant DMS-0456748}

\author{X.~Buff}

\author{A.~Ch\'eritat}

\author{L.~Oversteegen}

\thanks{The fourth author was partially  supported
by NSF grant DMS-0405774}

\address[A.~Blokh and L.~Oversteegen]
{Department of Mathematics\\ University of Alabama at Birmingham\\
Birmingham, AL 35294-1170, USA}

\address[X.~Buff and A.~Ch\'eritat]
{Institut de Math\'ematiques, Universit\'e Paul Sabatier, 118 route de Narbonne,
31062 Toulouse, France}

\email[A.~Blokh]{ablokh@math.uab.edu}
\email[L.~Oversteegen]{overstee@math.uab.edu}
\email[A.~Ch\'eritat]{cheritat@picard.ups-tlse.fr}
\email[X.~Buff]{buff@picard.ups-tlse.fr}



\begin{abstract} In general, little is known about the exact
topological structure of Julia sets containing a Cremer point. In this paper we
show that there exist quadratic Cremer Julia sets of positive area such that
for a full Lebesgue measure set of angles the impressions are degenerate, the
Julia set is connected im kleinen at the landing points of these rays, and
these points are contained in no other impression.
\end{abstract}

\maketitle

\section{Introduction}\label{intro}

A long outstanding problem of whether polynomial Julia sets of positive area
exist has recently been solved in \cite{bc06}. In that paper the authors
implemented a program initiated by A. Douady and showed that there exist
quadratic polynomials of various types all of which have Julia sets of positive
area. The proofs are based on McMullen's results \cite{mcm98} regarding the
measurable density of the filled-in Julia set near the boundary of a Siegel
disk with bounded type rotation number, Ch\'eritat's techniques of parabolic
explosion \cite{che01} and Yoccoz's renormalization techniques \cite{yoc95} to
control the shape of Siegel disks, and Inou and Shishikura's results \cite{is} to
control the post-critical sets of perturbations of polynomials having an
indifferent fixed point.

In particular, there exist quadratic polynomials with a fixed Cremer point
whose Julia set has positive area. The topological structure of Julia sets
containing a Cremer point (including the above mentioned ones of positive area)
has remained elusive. In the present paper we want to shed some light upon this
problem.

Let us first discuss the topological structure of connected polynomial Julia
sets in general. The best case scenario is when the Julia set $J$ is locally
connected. Then $J$ is homeomorphic to the quotient space $S^1/\sim=J_\sim$ of the unit circle
with respect to a specific equivalence relation $\sim$ called
an \emph{invariant lamination}. Moreover, the factor map in this case
conjugates the polynomial to the map $f_\sim$ induced on $J_\sim$ by the map
$z^d:S^1\to S^1$ with $d$ being the degree of the polynomial. Spaces like
$J_\sim$ are called below \emph{topological Julia sets} while the induced maps
$f_\sim$ on $J_\sim$ are called \emph{topological polynomials}. Thus, in the
locally connected case topological polynomials are good (one-to-one) models for
true complex polynomials on their Julia sets.

The picture is more complicated when $J$ is not locally connected. Still even
in that case recently developed tools allow one to use topological polynomials
as models (albeit not as good as above). Given a polynomial $P$, denote by
$J_P$ its Julia set. Call irrational neutral periodic points \emph{CS-points};
a CS-point $p$ is said to be a \emph{Cremer point} if the iterate of the map
which fixes $p$ is not linearizable in a small neighborhood of $p$. Suppose
that $P$ is a polynomial with connected Julia set and no CS-points. In his
fundamental paper \cite{kiwi97} Jan Kiwi introduced for $P$ an invariant
lamination $\sim_P$ on $S^1$ such that $P|_{J_{P}}$ is semi-conjugate to the
induced map $f_{\sim_P}:J_{\sim_P}\to J_{\sim_P}$ by a monotone map $m:J_P\to
J_{\sim_{P}}$ (by \emph{monotone} we mean a continuous map whose point
preimages are connected). Then $J_{\sim_P}$ is a locally connected model for
$J$, and $P|_J$ is monotonically semiconjugate to $f_{\sim_P}$. In addition
Kiwi proved in \cite{kiwi97} that at all periodic points $p$ of $P$ in $J_P$
the set $J_P$ is locally connected at $p$ and $m^{-1}\circ m(p)=\{p\}$.
However in some cases the entire approach which uses modeling of the Julia set
by means of a monotone map onto a locally connected continuum breaks down. From
now on we will consider only quadratic polynomials; the critical point will be
always denoted by $c$.

A quadratic polynomial $P$ with a \emph{Cremer \textbf{fixed} point} (i.e. with
a neutral non-linearizable fixed point $p\in J$ such that $P'(p)=e^{2\pi i
\al}$ with $\al$ irrational) is said to be a \emph{basic Cremer polynomial},
and its Julia set is called a \emph{basic Cremer Julia set} (in this case we
always denote the Cremer fixed point of $P$ by $p$). The main result of a
recent preprint \cite{bo06b} is that if $P$ is a basic Cremer polynomial then
any monotone map $m:J_P\to A$ with $m(J_P)$ locally connected must collapse all
of $J_P$ to a point. Hence studying the topology of basic Cremer Julia sets
requires new tools (see, e.g., the fundamental papers \cite{pere94} and
\cite{pere97} by Perez-Marco); some tools here are provided by continuum theory
and developed in \cite{grismayeover99} (in which results of \cite{yoc95} were
used) and further in \cite{bo06a}. Before introducing them let us have an
overview of a few known results.

By Sullivan \cite{sul83}, a basic Cremer Julia set $J$ is not locally
connected. Moreover, by Kiwi \cite{k00} the critical point $c$ is not
accessible. Still, there are points in $J$ at which rays are landing (e.g.,
repelling periodic points \cite{douahubb85}), so it makes sense to study in
more detail the pattern in which such landing can occur. In this respect the
following important question is due to C. McMullen \cite{mcm94}: can a basic
Cremer Julia set contain any points at which at least two rays are landing
(so-called \emph{biaccessible points})? This question was partially answered by
Schleicher and Zakeri in \cite[Theorem 3]{sz99} (see also \cite[Theorem
3]{zak00}) where they show that \emph{if} a basic Cremer Julia set contains a
biaccessible point then this point eventually maps to the Cremer point. In
\cite{bo06a} it is shown that if a basic Cremer Julia set has a biaccessible
point then it is a \emph{solar} Julia set as described in
Theorem~\ref{mainold}. However it is still unknown if there exist basic Cremer
Julia sets with biaccessible points. Another paper studying the topology of
basic Cremer Julia sets is that of S\o rensen \cite{sor98}. In that paper the
author constructs basic Cremer polynomials with external rays which accumulate
on both the Cremer point and its preimage and thus gives examples of basic
Cremer polynomials whose Julia sets have very interesting topological
properties.

Now we would like to state the results of \cite{bo06a}. If $P$ is a basic
Cremer polynomial then by \cite{mane93} $p\in \om(c)$ (see also \cite{pere97}
and \cite[Theorem 1.3]{chi05}). Following Kiwi~\cite{kiwi97} we say that two
angles $\ta, \ga$ are \emph{K-equivalent} if there are angles $x_0=\ta, \dots,
x_n=\ga$ such that the impressions of $x_{i-1}$ and $x_i$ are non-disjoint for
$1\le i\le n$; a class of K-equivalence is called a \emph{K-class}, and an
angle whose impression is disjoint from all other impressions is said to be
\emph{K-separate}. A continuum $X$ is \emph{connected im kleinen at a point
$x$} provided for each open set $U$ containing $x$ there exists a connected set
$C\subset U$ such that $x$ is in the interior of $C$ (relative to $X$). A continuum $X$ is
\emph{locally connected at a point $x$} provided there exists a basis of open
connected neighborhoods at $x$. Observe that sometimes different terminology is
used. For example, in Milnor's book \cite[p. 168]{miln00} the property of local
connectivity is called ``open local connectivity'' while to the property of
being connected im kleinen at a point Milnor refers to as the property of being
``locally connected at a point''. On the other hand, in the textbook by Munkres
\cite[p. 162]{mun00} connected im kleinen is called ``weakly locally
connected''. Using our terminology, if a space is locally connected at $x$,
then it is connected im kleinen at $x$. It is well known that if a continuum is
connected im kleinen at each point, then it is locally connected (see, e.g.,
\cite[p. 162, Ex. 6]{mun00}). However, a continuum can be connected im kleinen
at a point without being locally connected at this point (as an example one can
consider the so-called \emph{infinite broom}, see \cite[p. 162, Ex. 7]{mun00}).

The main result of \cite{bo06a} is the following theorem (by a
\emph{degenerate} impression we mean an impression consisting of one point).

\begin{theorem}\label{mainold} Let $P$ be a basic Cremer polynomial. Then
its Julia set $J$ must be one of the following two types.

\begin{description}

\item[Solar Julia set] $J$ has the following equivalent properties:

\begin{enumerate}

\item there is an impression not containing the Cremer point;

\item there is a degenerate impression;

\item the set $Y$ of all K-separate angles with degenerate
impressions contains all angles with dense orbits ($Y$ contains a full Lebesgue
measure $G_\delta$-set dense in $S^1$) and a dense in $S^1$ set of periodic
angles, and the Julia set $J$ is connected im kleinen at the landing points of
these rays;

\item there is a point at which the Julia set is connected im
kleinen;

\item not all angles are K-equivalent.

\end{enumerate}

\item[Red dwarf Julia set] $J$ has the following equivalent properties:

\begin{enumerate}

\item All impressions are non-degenerate.

\item The intersection of all impressions is a non-degenerate
subcontinuum of $J$ containing the Cremer point and the limit set of the
critical point.

\item $J$ is nowhere connected im kleinen.

\end{enumerate}

\end{description}

\end{theorem}

The main aim of this paper is to prove the following theorem.

\begin{theorem}\label{main} There exist basic Cremer polynomials
with solar Julia sets of positive area.
\end{theorem}

We prove Theorem~\ref{main} combining results from \cite{bc06} and
\cite{bo06a}. An interesting remaining problem then is that of the existence of
red dwarf Julia sets. It is related to a well-known problem concerning the
existence of \emph{indecomposable Julia sets} (a continuum is
\emph{indecomposable} if it cannot be represented as the union of two proper
subcontinua). It is known that if the Julia set is indecomposable then every
impression coincides with the entire Julia set (see, e.g., \cite{mr} and
\cite{cmr}).

\section{Preliminaries}

\subsection{General facts and notation}

In what follows we use standard tools of Carath\'eodory theory. An
\emph{unshielded} continuum $X\subset \C$ is a continuum which coincides with
the boundary of the infinite complementary component $U$ of $X$ in the complex
sphere. Given an unshielded continuum $X$ let $\phi$ be the normalized Riemann
map from the unit disk $\D$ onto $U$. The \emph{external ray $R_\al$} is the
image of the radius of $\D$ corresponding to the \emph{external} angle $\al$.
The \emph{impression $\imp(\al)$} of $R_\al$ is defined (see, e.g.,
\cite{pom92}) as the set of all limit points of sequences $\phi(x_i)$ taken
over all sequences $x_i\to e^{2\pi i\al}, x_i\in \D$ (if we do not want to
specify the angle we will omit it from the notation). By a \emph{component} of
a set we always mean a \emph{connected component} of this set, and by a
\emph{non-separating planar set} we mean a set whose complement in the plane is
connected. Also, given sets $A, B$ such that $A\cap B$ is a singleton, we shall
say that $A$ is \emph{attached} to $B$.

Given a polynomial $f:\C\to \C$ with non-separating Julia set $J(P)$ and a
continuum $K\subset J(P)$, a set $K''$ is said to be a \emph{$K$-pullback (of
order $n$ by $f$)} if $K''$ is a component of $f^{-n}(K)$. Clearly, any
$K$-pullback $K''$ maps onto $K$ as a branched covering map. For us the most
interesting is the case of quadratic polynomials with non-separating Julia sets
and subcontinua $K$ of the Julia set. Then if $K$ contains $f(c)$, the
first $K$-pullback is unique and maps onto $K$ in a 2-to-1 branched covering
fashion while if $K$ does not contain $f(c)$ then there are two first
$K$-pullbacks each of which maps onto $K$ homeomorphically. Given a quadratic
polynomial $f$ with the critical point $c_f=c$, we denote by $\psi_f=\psi$ the
involution which maps any $z$ to the other preimage of $f(z)$ (e.g., $\psi(c)=c$,
and if $f(z)=z^2+v$ then $\psi(z)=-z$). Below we often consider forward
invariant continua $K\subset J_f$ which contain $c$ but are such that the first
$n$-segment of the orbit of $c$ avoids $\psi(K)=K'$ (i.e., $f(c)\nin K', \dots,
f^n(c)\nin K'$). In this case we consider $K'$-pullbacks of order at most $n$
($K'$ is considered its own pullback of order $0$). Because of the assumptions,
all these $K'$-pullbacks map univalently onto $K'$.

\subsection{Siegel polynomials}\label{Siegel}

A quadratic polynomial is said to be a \emph{basic Siegel polynomial} if it has
an invariant Siegel disk. Denote by $\mathcal{S}_{lc}$ the family of basic
Siegel polynomials with locally connected Julia set. We need a few well-known
facts concerning their Julia sets $J$ (see, e.g., \cite{grismayeover99}). Given
a polynomial $P\in \mathcal{S}_{lc}$, let $\Da$ be its closed Siegel disk,
$S=\partial \Da$, and $\psi(\Da)=\Da'$. Since $J$ is locally connected,
$\{c\}=\Da'\cap \Da$, and so $\Da'$ is \emph{attached} to $\Da$ at $c$ and
contains no forward images of $c$. Consider the branch $c_{-1}, c_{-2}, \dots$
of the backward orbit of $c=c_0$ consisting only of points of $S$ (here
$P(c_{-n-1})=c_{-n}$). At each point $c_{-n}$ the appropriate $n$-th pullback
of $\Da'$, corresponding to $c_{-n}$ as the pullback of $c$, is attached to
$\Da$. However all other $\Da'$-pullbacks are disjoint from $\Da$. Since
\emph{all} forward images of $c$ avoid $\Da'$ then the picture described in the
previous subsection applies to $\Da$ and we have a family of well-defined
univalent $\Da'$-pullbacks of all orders. The entire set $P^{-n}(\Da)$ is a
connected union of $\Da$ and $2^n-1$ $\Da'$-pullbacks, $n$ of which are
attached to $\Da$, while others are disjoint from $\Da$.

\begin{figure}
\begin{picture}(307,231)
\put(0,0){\scalebox{0.5}{\includegraphics{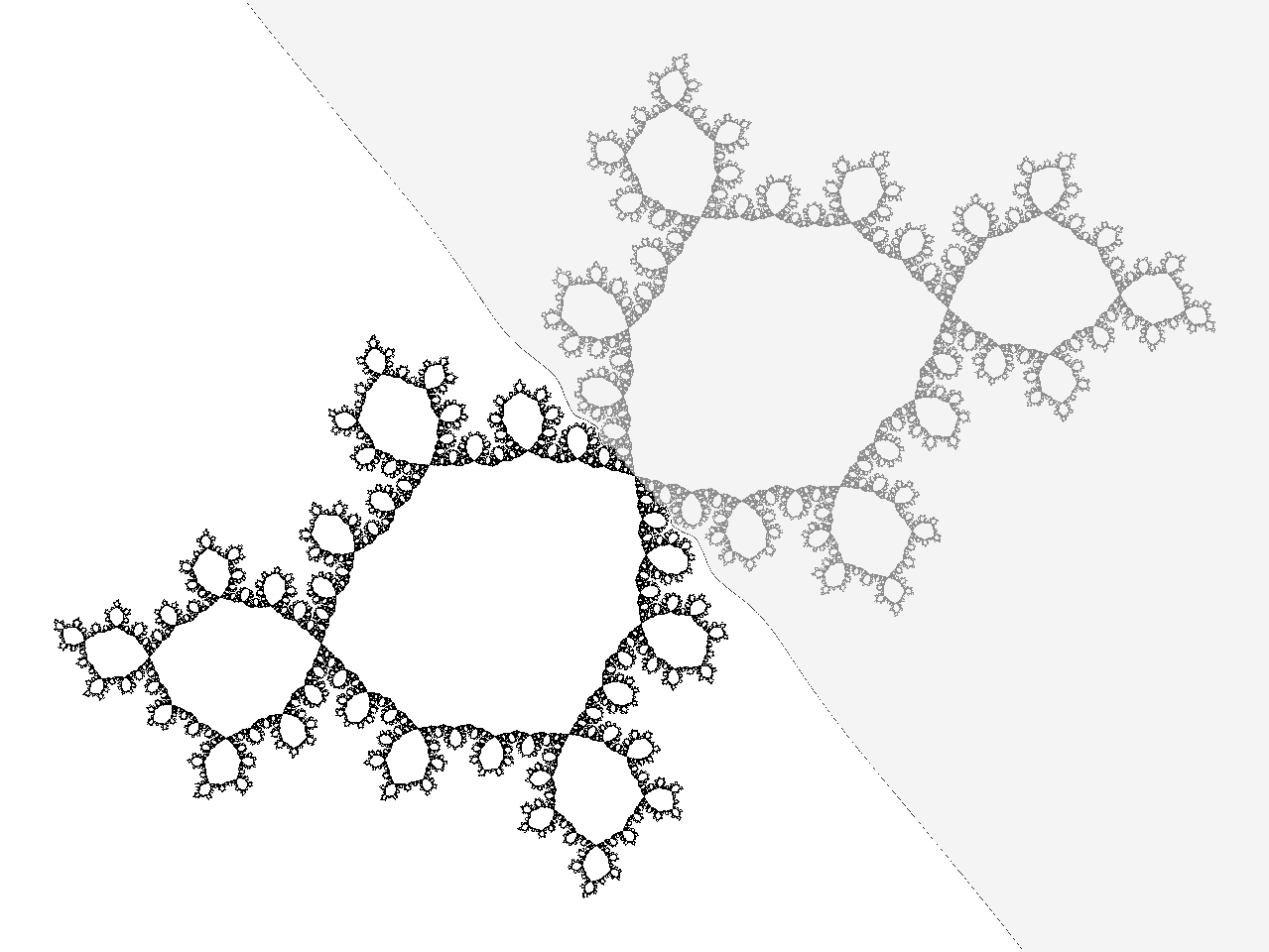}}}
\put(115,75){$\Delta$}
\put(183,146){$\Delta'$}
\put(220,10){$L$}
\put(250,10){$R$}
\end{picture}
\caption{Example of a locally connected basic Siegel polynomial Julia set: that with the golden mean rotation number. We put the two external rays landing at the critical point and grayed the half-plane $R$.}
\end{figure}
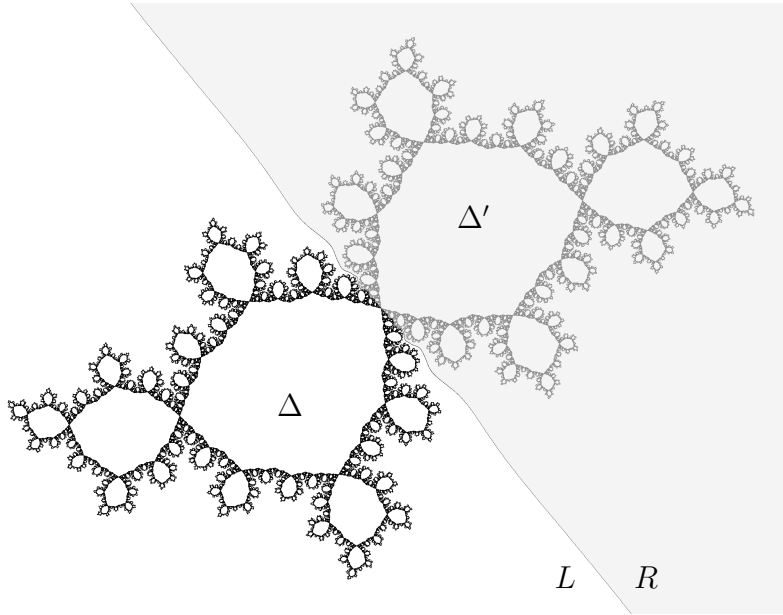

The way $\Da$ and $\Da'$-pullbacks intersect (but not their relative location
on the plane!) is the same for all polynomials $P\in \mathcal{S}_{lc}$. To
describe it, observe that two external rays landing at $c$ cut $\C$ into two
half-planes. Assign $1$ to the closed half-plane $L\supset \Da$ and $0$ to the
closed half-plane $R\supset \Da'$ and study the symbolic dynamics of points of $J$
in terms of this partition of the plane. To each point $x\in J$ we associate
its infinite itinerary $i(x)$ defined in the obvious way. Clearly, the only
ambiguity in $i(x)$ arises if the point $x$ is a critical preimage because the
only point of $J$ which belongs to both $L$ and $R$ is $c$ (and hence $c$ can
be assigned both $0$ and $1$ as the first entry in its itinerary). This
ambiguity is resolved though if instead of points we deal with
$\Da'$-pullbacks. Namely, if $Q$ is $\Da$ or a $\Da'$-pullback then we assign
as its infinite itinerary $i(Q)$ the itinerary of any non-precritical point of
$Q$ (any such itinerary $i(Q)$ from some time on consists of $1$'s).  To
simplify the notation, let us denote by $1_i$ the string of 1's of length $i$
(possibly, $i=\iy$). Similarly, if a finite string $i'=i_0i_1\dots i_l$ is
given then $i'_k=\{i_0\dots i_l\}_k$ is the concatenation of $k$ copies of $i'$
(here again $k$ can be $\iy$). Then we have, e.g., that $i(\Da)=1_\iy$,
$i(\Da')=01_\iy$, etc.

\begin{figure}
\begin{picture}(306,256)
\put(0,0){\scalebox{0.5}{\includegraphics{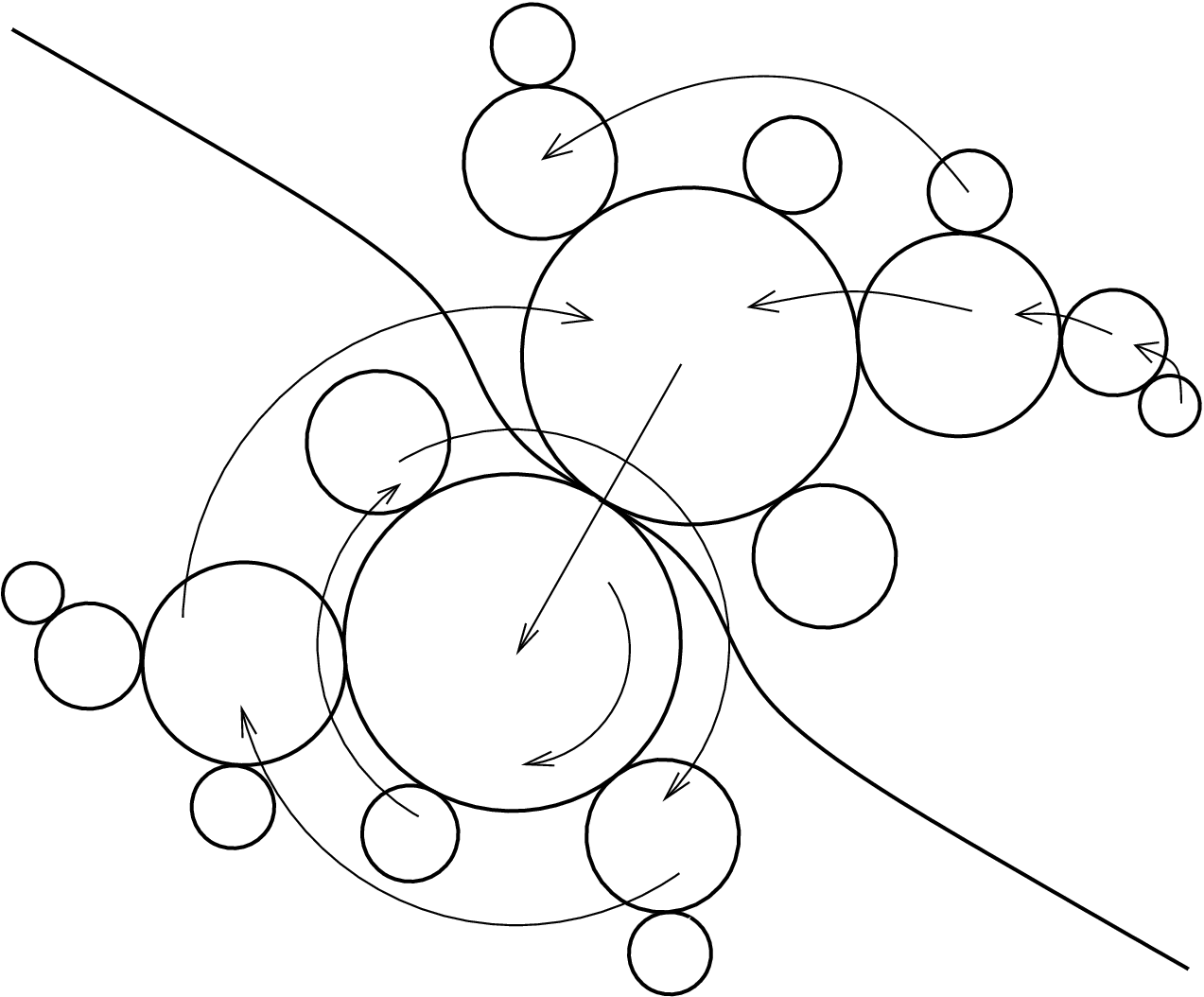}}}
\put(250,15){$L$}
\put(270,35){$R$}
\put(110,80){$1_\infty$}
\put(163,172){$01_\infty$}
\put(235,158){$001_\infty$}
\put(285,185){$0001_\infty$}
\put(50,82){$101_\infty$}
\put(153,37){$1101_\infty$}
\put(82,143){$\scriptstyle 11101_\infty$}
\put(197,110){$\scriptstyle 01101_\infty$}
\put(121,206){$0101_\infty$}
\put(242,220){$\scriptstyle 00101_\infty$}
\put(11,84){$\scriptstyle 1001_\infty$}
\put(185,7){$\scriptstyle 11001_\infty$}
\put(150,240){$\scriptstyle 01001_\infty$}
\put(35,28){$\scriptstyle 10101_\infty$}
\put(282,130){$\scriptstyle 00001_\infty$}
\put(3,115){$\scriptstyle 10001_\infty$}
\end{picture}
\caption{Example of itineraries of $\Delta$ and of pullbacks of $\Delta'$ and of their dynamics for the Golden mean basic Siegel Julia set. Some of the missing arrows can be determined by using the central symmetry: two symmetric components are mapped to the same component.}
\end{figure}
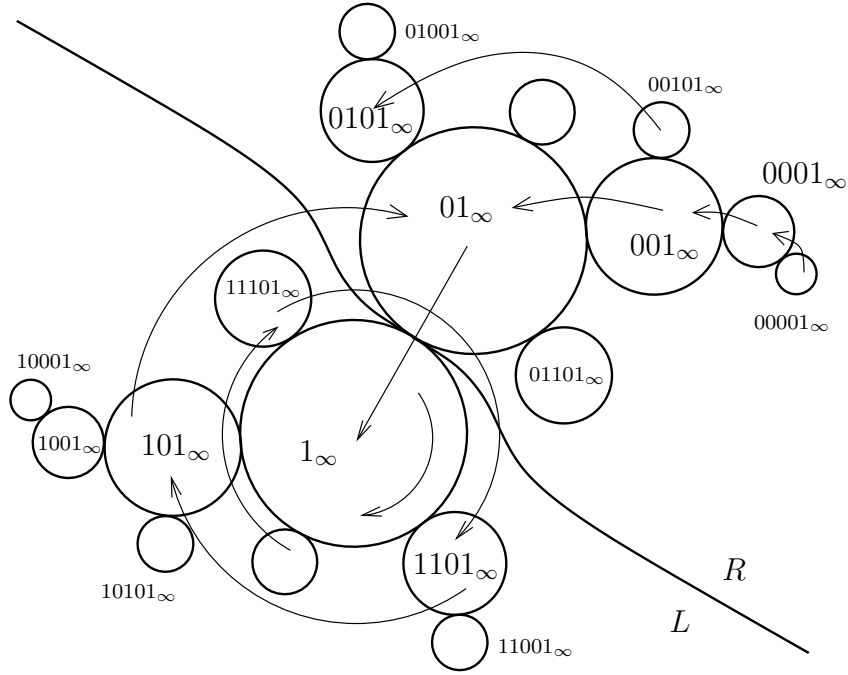

For $c$ and its preimages the above ambiguity can be dealt with by considering
$c$ as a point of $\Da$ or as a point of $\Da'$ and assigning different
itineraries respectively. That is, if we take a preimage of $c$ then until it
maps onto $c$ its itinerary is well-defined but at the time when it maps onto
$c$ we either consider it as a point of $\Da$ and assign $1$ to it, or consider
it as a point of $\Da'$ and assign $0$ to it. For the forthcoming images of $c$
we again have no ambiguity ($1$ is assigned to all of them). In other words, to
each preimage of $c$ exactly two itineraries are assigned as explained above.
This implies that two $\Da'$-pullbacks with itineraries $\bar i=i_0i_1\dots$
and $\bar j=j_0j_1\dots$ have a point in common if and only if there exists $k$
such that $i_0=j_0, \dots, i_{k-1}=j_{k-1}, i_k=1, j_k=0$ while $i_r=j_r=1$ for
any $r>k$ (in particular, no more than two $\Da'$-pullbacks can intersect at
one point). In what follows we denote the $\Da'$-pullback with itinerary $\bar
i$ as $\Da_{\bar i}$. Observe that since $\Da$ and $\Da'$ have only one point -
namely $c$ - in common, then any two $\Da'$-pullbacks (or $\Da$) may have at
most one point in common and no three $\Da'$-pullbacks (or $\Da$) intersect
(any common point of two $\Da'$-pullbacks maps onto $c$ while the two pullbacks
in question map onto $\Da$ and $\Da'$).

So far we have considered as examples points of $\Da'$-pullbacks (or
$\Da'$-pullbacks themselves). To deal with other points of $J$ we need
well-known facts about Julia sets of polynomials $P\in \mathcal{S}_{lc}$ (see,
e.g., \cite{grismayeover99}); the notation introduced here will be used from
now on. Let $p$ be the fixed point of $P$ belonging to $\Da$, $\si$ be the
angle doubling map on $S^1$, $(\ga, \ga')$ (resp. $[\ga, \ga']$) be the open
(resp. closed) counterclockwise circle arc from $\ga$ to $\ga'$, and let
$\ol{\ga\ga'}$ be the chord connecting $\ga$ and $\ga'$ in the disk. If
$P'(p)=e^{2\pi i\rho}$ ($\rho$ is irrational) then there exists a special \emph{rotational}
$\si$-invariant Cantor set $F\subset S^1$ such that $\si|_F$ is no more than
$2$-to-$1$ semiconjugate
to the irrational rotation by the angle $2\pi \rho$ \cite{bullsent94}. More precisely,
call an arc complementary to $F$ an \emph{$F$-hole}. Then the semiconjugacy
maps the endpoints of every $F$-hole into one
point and otherwise is one-to-one. The most important $F$-hole is the longest
one which is the half-circle $(\be, \al)$ with the endpoints denoted below by
$0<\al<1/2$ and $\be=\al+1/2$; the chord $\ol{\al \be}$ connecting $\al$ and
$\be$ is called the \emph{critical leaf (diameter)}. Other $F$-holes are
preimages of $(\be, \al)$.

The limit set $F=\om(\al)$ is exactly the set of angles whose entire orbits are
contained in $[\al, \be]$; also, the angles in $F$ are exactly the angles whose
rays land at points of $S=\partial \Da$. The endpoints of an $F$-hole can be
denoted by $\al_{-n}, \be_{-n}$ (since they are appropriate preimages of $\al$
and $\be$). Both rays $R_{\al_{-n}}$ and $R_{\be_{-n}}$ land at the point
$c_{-n}$, the $n$-th pullback of $c$ belonging to $\Da$. We need the following
simple fact.

\begin{lemma}\label{separ} Suppose that two
angles $\al''\ne \be''$ are given none of which maps into $F$ by any power of
$\si$. Then there exists $m$ such that $\si^m(\al'')$ and $\si^m(\be'')$ are
separated by the critical leaf $\ol{\al \be}$.
\end{lemma}

\begin{proof} Define $d(\ta, \ta')$ as the length of the shortest
arc between $\ta$ and $\ta'$ (we normalize the circle so that its length is
equal to $1$). It is easy to see that $d(\si(\ta), \si(\ta'))=T(d(\ta, \ta'))$
where $T:[0, 1/2]\to [0, 1/2]$ is the appropriate scaling of the full tent map.
Since $T(x)>x$ for $0<x<1/3$, there exists $m$ such that $d(\si^m(\al''),
\si^m(\be''))\ge 1/3$. If $d(\si^m(\al''), \si^m(\be''))=1/2$ then since
$\al'', \be''$ are not preimages of $\si(\al)$ we see that $\si^m(\al''),
\si^m(\be'')$ are separated by $\ol{\al \be}$ and we are done. Assume that
$\si^m(\al''), \si^m(\be'')$ are not separated by $\ol{\al \be}$ and
$d(\si^m(\al''), \si^m(\be''))<1/2$. Since the longest complementary arcs to
the set $\si^{-1}(F)$ are of length $1/4$, we see that $\si^m(\al''),
\si^m(\be'')$ belong to two distinct complementary arcs of $F\cup F+1/2$
located on one side of $\ol{\al \be}$. Thus after several steps the $\si$-image
of, say, $\al''$ will belong to $(\be, \al)$ while the corresponding image of
$\be''$ will still be inside $(\al, \be)$. This means that these two images of
$\al'', \be''$ will be separated by $\ol{\al \be}$.
\end{proof}

Call a point $y\in J$ a \emph{local cutpoint} of $J$ if the point $y$ is a
cutpoint of some connected neighborhood $U$ of $y$ in $J$. The union of
boundaries of $\Da$ and all $\Da'$-pullbacks forms the set of all local
cutpoints of $J$. The remaining points of $J$ are called the \emph{endpoints}
of $J$, or, more informally, the ``dust''. We say that a connected set $X$ \emph{connects}
a connected set $A$ and a connected set $B$ if the union $A\cup B\cup X$ is connected (in
practice we use this term when $A$ and $B$ are disjoint). Also, by a \emph{string} of
$\Da'$-pullbacks we mean a countable collection of $\Da'$-pullbacks concatenated
to each other (so that consecutive pullbacks in the string intersect over exactly one point).
Lemma~\ref{sy} studies how dust points in $J$ are connected to $\Da$.

\begin{lemma}\label{sy} Let $y\in J$ be a dust point with itinerary $\bar i=(i_0i_1\dots)$.
Then there exists a unique string $S_y$ of $\Da'$-pullbacks which connects $\Da$ and
$y$. Denote the $\Da'$-pullbacks in $S_y$ as follows: $\Da^1(y)$ is the closest to $\Da$
(in the sense of the spatial order on the string), $\Da^2(y)$ is the second one, etc.
Then the itinerary of $\Da^j(y)$ is obtained
from the itinerary $\bar i=i_0i_1\dots$ of $y$ as follows: choose the $j$-th
appearance of $0$ in $\bar i$, keep all the entries before that, and replace
all other entries in $\bar i$ by $1$.
\end{lemma}

\begin{proof} Since $J$ is locally connected, it is arcwise connected. Hence there exists
an \emph{arc} (homeomorphic image of the interval $[0, 1]$) $I'\subset J$ connecting $y$ and $c$.
Let us show that $I'\cap \Da=I'\cap S$ is connected (recall, that $S=\partial \Da$). Indeed,
suppose that there exists an arc in $J$ whose endpoints belong to $S$ while otherwise the arc
is disjoint from $S$. Then it follows that parts of $S$ are ``shielded'' from infinity
by this arc, i.e. are not accessible from infinity, a contradiction. Hence such arcs
in $J$ do not exist which shows that $I'\cap S$ is either an arc or the point $c$.
If $I'\cap S=\{c\}$ is a point set $\pi(y)=c$, otherwise let $\pi(y)$
be the other endpoint of $I'\cap S$. Observe that by the above argument such point is
unique so that $\pi$ is well-defined. Denote the arc connecting $y$ and $\pi(y)$ by $I$;
then $I\cap S=\pi(y)$.
Intuitively, one can think of the point $\pi(y)$ as a ``projection'' of $y$ into $S$.

Since all points of $I$ except for $y$ are not from the dust,
it follows that $I\setminus \{y, \pi(y)\}$ is contained in the union of some $\Da'$-pullbacks.
The union of these $\Da'$-pullbacks forms the desired string $S_y$. Observe, that
the string $S_y$ is unique by the same geometric argument as above - otherwise some
points of the Julia set are not
accessible from infinity because they will be ``shielded'' from infinity by other
connected subsets of $J$ (these subsets will be parts of the two strings which
hypothetically connect $\Da$ and $y$ and will form the boundary of a simply connected
domain in the plane containing points of $J$). A similar geometric argument is often used
in the paper, so we explain it here in detail while simply alluding to it in the future.

By the definition we do \textbf{not} include
$\Da$ in $S_y$ and begin $S_y$ from $\Da^1(y)$, the $\Da'$-pullback
closest to $\Da$ in $S_y$. Thus, $\Da^1(y)$ is
either $\Da'$ or a pullback of $\Da'$ attached to $\Da$. Therefore $\Da^1(y)$
has itinerary $1_{k_1}01_\iy$ with $k_1\ge 0$ digits 1 to begin with (if
$k_1=0$ then $\Da^1(y)=\Da'$). The next pullback $\Da^2(y)$ in the string $S_y$
is attached to $\Da^1(y)$, hence its itinerary coincides with that of
$\Da^1(y)$ until $\Da^1(y)$ maps onto $\Da$. At this moment $\Da^2(y)$ becomes
the closest to $\Da$ pullback of $\Da'$ in the appropriate image of $S_y$ and
the process repeats itself with the only difference that now the image of
$\Da^2(y)$ needs $k_2$ steps to move around $\Da$ until it finally gets mapped
onto $\Da'$ and then onto $\Da$. This argument yields that the
``spatially'' $j$-th pullback $\Da^j(y)$ of $\Da'$ in $S_y$ has itinerary
$1_{k_1} 0 1_{k_2} 0 \dots 1_{k_j} 0 1_{\iy}$. It has to coincide with the
itinerary of $y$ until the first time $\Da^j(y)$ maps onto $\Da$. Also, numbers
$k_i$ may be equal to zero. Hence the itinerary of $\Da^j(y)$ is obtained
from the itinerary $\bar i=i_0i_1\dots$ of $y$ as follows: choose the $j$-th
appearance of $0$ in $\bar i$, keep all the entries before that, and replace
all other entries in $\bar i$ by $1$.
\end{proof}

As an example consider the periodic point $x$ with itinerary $\bar
i=\{011\}_\infty$. Then the following are the pullbacks of $\Da'$ forming the
string $S_x$: $\Da^1(x)=\Da_{01_\iy}, \Da^2(x)=\Da_{01101_\iy},
\Da^3(x)=\Da_{01101101_\iy}$ etc. By Lemma~\ref{separ} any two distinct
periodic points $u\in J$ and $v\in J$ have distinct itineraries (it is enough
to consider  rays landing at $u$ and $v$). Therefore their strings $S_u$ and
$S_v$ may have a certain initial piece in common (perhaps empty), and then will
separate. For example, let $z$ be the periodic point of itinerary
$\{0110111\}_\iy$. Then the string $S_z$ consists of the following pullbacks of
$\Da'$: $\Da^1(z)=\Da_{01_\iy}, \Da^2(z)=\Da_{01101_\iy},
\Da^3(z)=\Da_{011011101_\iy}, \dots$. Comparing $S_x$ and $S_z$ we see that
$\Da^1(x)=\Da^1(z), \Da^2(x)=\Da^2(z)$, but $\Da^3(x)\ne \Da^3(z)$. It follows
that $\Da^3(x)$ and $\Da^3(z)$ are disjoint (both are attached to the same
$\Da'$-pullback but at distinct points, hence if they meet then some points of
the Julia set will not be accessible from infinity).
From this moment on the strings $S_x$ and $S_z$ go their own ways,
converging to $x$ and $z$ respectively.

Consider now the dynamics of the string $S_x$. Since $x$ is of period $3$ then
$S_x$ must cover itself under $P^3$ while $P^3(S_x)$ contains $\Da$. In fact,
already the first application of $P$ restricted onto $S_x$ maps $\Da^1(x)$ onto
$\Da$. Then $S_x$ rotates about $\Da$ for one step and on the next step $S_x$
maps over itself. Thus, $P^3(\Da^2(x))=\Da^1(x), P^3(\Da^3(x))=\Da^2(x)$ etc.
In other words, $P^3$ shifts the pullbacks in $S_x$\,\,$1$ pullback ``down'' (i.e.
closer to $\Da$). The number $1$ is then called the \emph{basic length} of
$S_x$; the string $S_x$ consists of countably many fragments whose length is 1
and who are shifted by $P^3$ one onto another closer to $\Da$ except for the
first fragment of $S_x$ which maps by $P^3$ onto $\Da$.

In the case of the string $S_z$ the picture is a bit more complicated, however
it has essentially the same properties (in our description of the dynamics of
$S_z$ we skip discussing simple rotations of $S_z$ around $\Da$). It is easy to
check that first $P^3$ shifts $\Da^2(z)$ onto $\Da^1(z)$ while
$\Da^3(z)=\Da_{011011101_\iy}$ maps onto $\Da_{011101_\iy}$. Then $P^4$ maps
$\Da_{011101_\iy}$ onto $\Da'=\Da^1(z)$ and the entire string $S_z$ finally
covers itself. In other words, since the periodic fragment of the itinerary of
$z$ contains two zeros, then $P^7$ shifts the pullbacks in $S^z$ down
(closer to $\Da$) by $2$ pullbacks. In this case the basic length of $S_z$ is $2$; the string
$S_z$ consists of countably many fragments whose length is 2 and who are
shifted by $P^7$ one onto another closer to $\Da$ except for the first fragment
of $S_z$ which maps by $P^7$ onto $\Da$. Clearly, $P^7$ maps the first fragment
of $S_z$ onto $\Da$ as a continuous map (there have to be critical points), but
otherwise the map $P^7$ shifts fragments in $S_z$ homeomorphically.

In general, given a periodic orbit $y$ of period $k$ whose itinerary has the
minimal periodic fragment with $l$ zeros we see that the map $P^k$ shifts the
$\Da'$-pullbacks in $S_y$ down (closer to $\Da$) by $l$ pullbacks. In this case
the first fragment of $S_y$ is $\Da^1(y)\cup \dots \cup \Da^l(y)$, the second
fragment of $S_y$ is the union $\Da^{l+1}(y)\cup \dots \cup \Da^{2l}(y)$, etc.
Observe that all the fragments in $S_y$ are in fact the pullbacks of the first
fragment by the appropriate branch of the inverse function to $P^l$, and from
the second fragment on all the fragments in $S_y$ are disjoint from $\Da$.
Moreover, since in this case the critical limit set $\omega_P(c)$ coincides
with $\Da$, we see that by well-known shrinking properties of pullbacks under
polynomial/rational maps (see, e.g., Shrinking Lemma \cite{lyumin}) the
diameters of the pullbacks in $S_y$ converge to $0$. If diameters of the sets
in a sequence converge to $0$ then the sets are said to form a \emph{null
sequence}; it follows that in this case the fragments in $S_y$ described above
form a \emph{null sequence}. Of course, in the case at hand \emph{any} string
$S_\zeta$ converges to the point $\zeta$ defining it - after all, the Julia set
$J$ is locally connected. However even in the case of basic Cremer Julia sets
(which are not locally connected) these ideas, with some modifications, still
apply.

\newcommand{\mycbrace}[2]{%
%\psline(#2,#2)(!#2\space #1\space add\space #2)
 \rput(#2,#2){\psline(#1,0)}%
 \rput(-#2,#2){\psline(-#1,0)}%
 \rput(#2,0){\psarc{-}{#2}{90}{180}}%
 \rput(-#2,0){\psarc{-}{#2}{0}{90}}%
 \rput(#1,#2){\rput(#2,#2){\psarc{-}{#2}{-90}{0}}}%
 \rput(-#1,#2){\rput(-#2,#2){\psarc{-}{#2}{-180}{-90}}}%
}
\begin{figure}
\begin{picture}(360,300)
\put(0,0){\scalebox{0.75}{\includegraphics{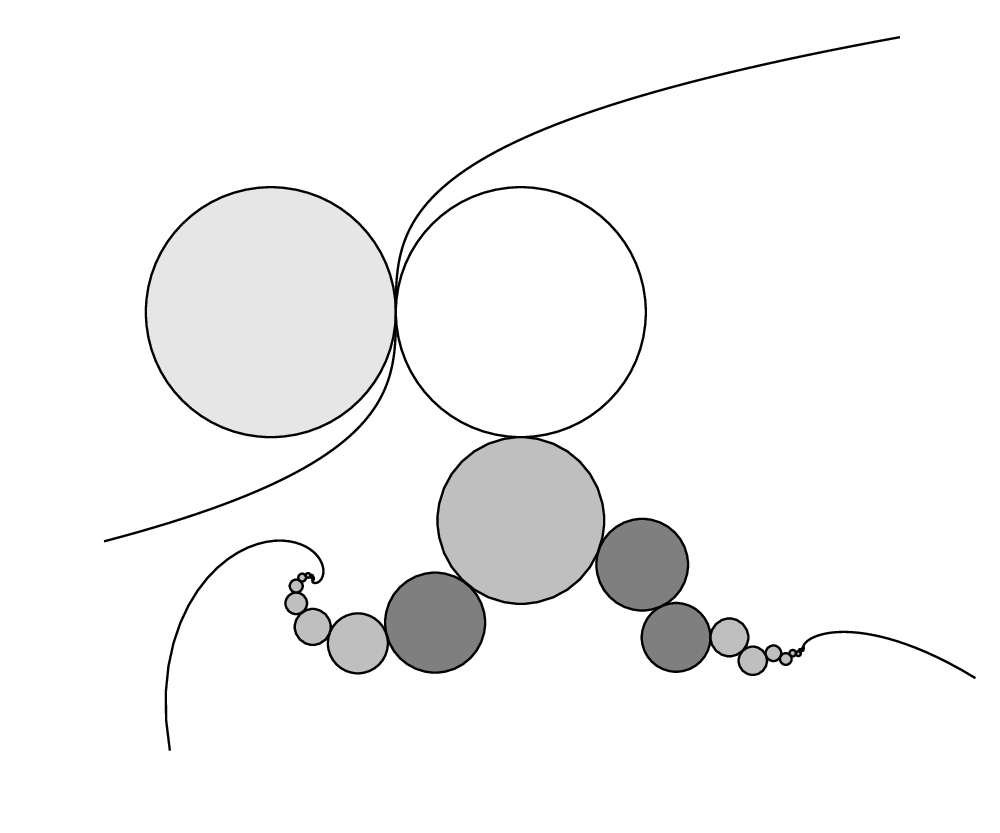}}}
\put(93,185){$\Delta$}
\put(184,185){$\Delta'$}
\put(163,111){$\scriptstyle\Da^2(x)=\Da^2(z)$}
\put(121,90){$x$}
\put(40,35){$R_\gamma$}
\put(295,61){$z$}
\put(335,45){$R_{\gamma'}$}
\put(146,50){\psarc{->}{0.6}{-180}{0}}
\put(141,20){$P^3$}
\put(269,84){%
 \put(0,0){\psarc{->}{0.4}{-45}{135}}%
 \put(10,10){$P^7$}%
 \put(-9,8){\rput{110}{\mycbrace{0.25}{0.1}}}
 \put(7,-9){\rput{130}{\mycbrace{0.10}{0.05}}}
}
\end{picture}
\caption{Illustration, for the golden mean basic Siegel Julia set, of the examples of periodic points $x$ and $z$ considered in the text, together with the associated strings of $\Delta$-pullbacks.}
\end{figure}
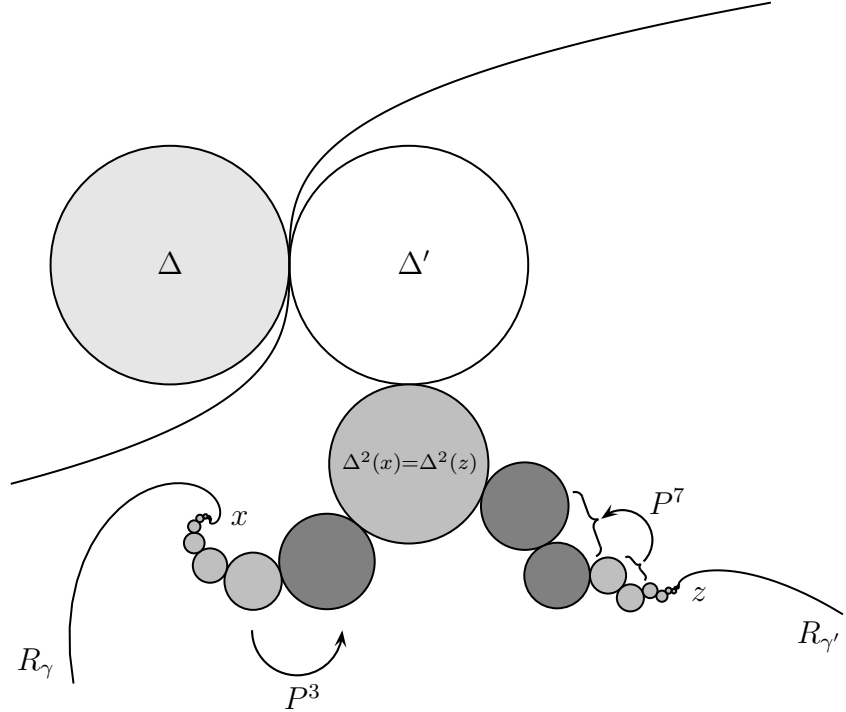

Let us go back to the example we have already partially considered before, i.e.
to the periodic points $x, z$ and their strings $S_x, S_z$. As we have seen,
$S_x$ and $S_z$ converge to the periodic points $x$ and $z$ respectively. On
the other hand, we have seen that the strings $S_x$ and $S_z$ separate after
the pullback $\Da^2(x)=\Da^2(z)=\Da_{01101_\infty}$. Let $R_\ga, R_{\ga'}$ be
the external rays landing at $x$ and $z$ respectively. Denote by $I$ the open
arc between $\ga$ and $\ga'$ contained in the $0$-semicircle. Denote by $\hat
S_x$ the closure of the ``tail'' of $S_x$ taken from $\Da_{01101\infty}$ to the
point $x$ and by $\hat S_z$ the closure of the ``tail'' of $S_z$ taken from
$\Da_{01101\infty}$ to the point $z$ (so that $\hat S_x\cap \hat
S_z=\Da_{01101\infty}$). Then the union of $\hat S_x, \hat S_z, R_\ga$ and
$R_\ga'$ is a ``fork'' which encloses a wedge containing all rays $R_\ta,
\ta\in I$. The impression of $\ta\in I$ must be contained in the closure of
this wedge and therefore is disjoint from $\Da$. Again, for the polynomials in
$\mathcal{S}_{lc}$ this fact can be shown in a much easier way. However the
argument is valuable because it applies in more general situations.

Let $\si_2$ be the one-sided shift on the space of all
sequences with symbols $1$ and $0$, and let $A_n=\si_2^{-n}(1_\iy)$. Notice
that $A_n$ includes all $\si_2$-preimages of $1_\iy$ of order less than or
equal to $n$. Consider elements of $A_n$ as vertices of a graph connected by
edges if and only if the $\Da'$-pullbacks with these itineraries intersect.
Clearly, this makes $A_n$ into a tree (if there is a loop in $A_n$ then as before there
must be points of $J$ shielded from the infinity, a conradiction).
From now on we always consider $A_n$
endowed with the tree structure. Suppose that for each $\ol{s}\in A_n$ a set
$M_{\ol{s}}$ is given and $g:\cup M_{\ol{t}}\to \cup M_{\ol{t}}$ is a function
such that 1) $g(M_{\ol{s}})\subset M_{\si_2(\ol{s})}$, and 2) $M_{\ol{s}}\cap
M_{\ol{t}}\ne\0$ if and only if $\Da_{\ol{s}}\cap \Da_{\ol{t}}\ne \0$. Then we
say that $(g, \{M_{\ol{s}}, \ol{s}\in A_n\})$ has the same \emph{(dynamical)
intersection pattern} as $\{\Da_{\ol{s}}, \ol{s}\in A_n\}$. It follows from the
results of \cite{bc06} (see below) that for certain basic Cremer polynomials
$f$ there is an invariant continuum $M$ which plays the role of the Siegel disk
$\Da$ above in the sense that the intersection pattern of $M$ and all
appropriately defined $M'$-pullbacks of order $n$ and the intersection pattern
of $\Da$ and all $\Da'$-pullbacks of order $n$ are the same.

\subsection{Basic Cremer polynomials}

Let us introduce terminology and notation. Let $P_\al(z)=e^{2\pi i\al}z+z^2$.
Clearly, the critical point of $P_\al$ is $c_\al=(-.5)e^{2\pi i\al}$; set
$M_\al=\om(c_\al)$ and call it the \emph{critical limit set} of $P_\al$. It is
well-known (see, e.g., \cite{mane93}) that if $P_\al$ is a basic Cremer
polynomial then its Cremer fixed point $p_\al$ belongs to $M_\al$. The next
lemma is obtained in \cite{chi05}.

\begin{lemma}[Childers]\label{conem} If $P_\al$ is a basic Cremer polynomial then
the set $M_\al$ is a continuum.
\end{lemma}

The construction in the next section aims at proving that certain Julia sets of
positive area are solar. We need the following definition: an angle $\ta$ is
said to be of \emph{bounded type} if there exists $K$ such that in the
continued fraction expansion $[a_1, a_2, \dots]$ of $\ta$ we have $a_i<K$ for
all $i$. Consider the polynomial $P_\ta(z)$ assuming that $\ta$ is of bounded
type. Then $P_\ta$ has an invariant closed Siegel disk $\Da_\ta$, $\partial
\Da_\ta=M_\ta$ is a simple closed curve and $J_{P_\ta}$ is locally connected
\cite{pet96}. For every $N$, define $\tS_N$ as the set of all $\ta$'s with
$a_i\ge N, i=1,  2, \dots$. Given two sets $A, B$ let us denote by $\partial[A,
B]$ the number $\sup \{\dist(a, B): a\in A\}$ (thus, if $\partial[A, B]=\e$
then the set $A$ is contained in the closed $\e$-neighborhood of $B$).
Theorem~\ref{bc} is proven in \cite{bc06}. The key tool in the proof is the use
of Inou-Shishikura's results on near parabolic renormalization \cite{is}.

\begin{theorem}[Buff-Ch\'eritat]\label{bc} There exists $N$ such that for every
$\ta\in \tS_N$ of bounded type and every $\e>0$ if $\al\in \tS_N$ is
sufficiently close to $\ta$ then $\partial [M_{\al}, \Da_\ta]<\e$ and
$M_{\al}\ne J_{P_{\al}}$. Moreover, there exist basic Cremer polynomials
$P_\al$ with $\al\in \tS_N$ arbitrarily close to $\ta$ and with positive area
Julia sets.
\end{theorem}

In what follows we will use the notation introduced above.

\section{Main theorem}

The aim of this section is to prove Theorem~\ref{main}. From now on we fix $N$
%and $\ta\in S_N$
as in Theorem~\ref{bc}. For a polynomial $P_\al$, we denote by $\psi_\al$ the
involution which maps any $z$ to the other preimage of $P_\al(z)$. Set
$\psi_\al(M_\al)=M'_\al$. The following lemma uses self-explanatory notation.

\begin{lemma}\label{inter} Given $n$, $\ta\in \tS_N$ of bounded type, and $\e>0$ there exists a
neighborhood $U$ of $\ta$ such that for any $\al\in \tS_N\cap U$ we have
$P_\al^i(c_\al)\nin M'_\al, i\le n$, and there exists a collection of
well-defined $M'_\al$-pullbacks $M_{\al, \ol{j}}, \ol{j}\in A_n$ by $P_\al$
such that $\partial [M_{\al, \ol{j}}, \Da_{\ta, \ol{j}}]<\e$ for any $\ol{j}\in
A_n$ and the $M'_\al$-pullbacks $\{M_{\al, \ol{j}}, \ol{j}\in A_n\}$ have the
same intersection pattern as $\Da'_\ta$-pullbacks $\{\Da_{\ta, \ol{j}},
\ol{j}\in A_n\}$.
\end{lemma}

\begin{proof} The proof is based upon Theorem~\ref{bc} and continuity arguments.
By Theorem~\ref{bc}, $M'_\al$ is very close to $\Da'_\ta$ for $\al\in \tS_N$
sufficiently close to $\ta$. Since $P^i_\ta(c_\ta)\nin \Da'_\ta$, it follows
then that $P^i_\al(c_\al)\nin M'_\al, i\le n$ for $\al\in \tS_N$ sufficiently
close to $\ta$ (so that an $i$-th pullback of $M'_\al$ is mapped 1-to-1 if
$i\le n$). Moreover, by continuity and by Theorem~\ref{bc} given an itinerary
$\ol{j}\in A_n$ we can correctly define $M'_\al$-pullbacks $M_{\al, \ol{j}}$ by
$P_\al$ corresponding to $\Da'_\ta$-pullbacks $\Da_{\ta, \ol{j}}$ by $P_\ta$
and guarantee that 1) $\partial [M_{\al, \ol{j}}, \Da_{\ta, \ol{j}}]<\e$ for
any $\ol{j}\in A_n$, and hence 2) for two itineraries $\ol{s}\in A_n, \ol{t}\in A_n$
if $\Da_{\ta, \ol{s}}\cap \Da_{\ta, \ol{t}}=\0$ then $M_{\al, \ol{s}}\cap
M_{\al, \ol{t}}=\0$. Let us show that if $\Da_{\ta, \ol{s}}\cap \Da_{\ta,
\ol{t}}\ne \0$ then $M_{\al, \ol{s}}\cap M_{\al, \ol{t}}\ne \0$. Observe that
since $A_n$ is a tree then if we remove the edge in $A_n$ connecting $\ol{s}$
and $\ol{t}$ then $A_n$ falls into two trees, $A_{\ol{s}}$ and $A_{\ol{t}}$.
Now, it is easy to see that the union of all pullbacks $M_{\al, \ol{r}},
\ol{r}\in A_n$ is the set $P_\al^{-n}(M_\al)$ which is connected because the
entire orbit of $c_\al$ is contained in $M_\al$ ($c_\al$ is recurrent for
$P_\al$ and $M_\al=\om(c_\al)$). On the other hand suppose that $M_{\al,
\ol{s}}\cap M_{\al, \ol{t}}=\0$. Then the unions $M_{\al,
A_{\ol{s}}}=\cup_{\ol{i}\in A_{\ol{s}}}M_{\al, \ol{i}}$ and $M_{\al,
A_{\ol{t}}}=\cup_{\ol{i}\in A_{\ol{t}}}M_{\al, \ol{i}}$ are disjoint because
they could only intersect over $M_{\al, \ol{s}}\cap M_{\al, \ol{t}}$ which is
empty. This contradiction shows that $M_{\al, \ol{s}}\cap M_{\al, \ol{t}}\ne
\0$ as desired.
\end{proof}

We are now ready to prove Theorem~\ref{main}. However it will be convenient to
first prove a simple technical lemma.

\begin{lemma}\label{pullba} Let $P$ be a basic Cremer polynomial and
$K\subset J$ be a continuum such that for some $k>0$ we have that
$P^{-k}(K)\cap K\ne \0$ and $P^i(c)\nin P^k(K)$ for any $i\ge 0$. Then
there exists a unique sequence $K(0)=K, K(-1), \dots$ of $K$ pullbacks by $P^k$
such that $K(-i)\cap K(-i-1)\ne \0$ and $P^k(K(-i-1))=K(-i)$ for all $i\ge 0$.
Moreover, if $P^{2k}(K)$ is disjoint from $K$ then in the sequence of sets
$P^{2k}(K), P^k(K), K, \dots$  non-empty intersections are only possible
between two consecutive sets.
\end{lemma}

\begin{proof} Since $J$ is a non-separating one-dimensional continuum, $J$
is a tree-like continuum (i.e., for each $\e>0$ there exists a map $\phi:J\to T$
where $T$ is a finite tree such that for each $t\in T$ the diameter of $\phi^{-1}(t)<\e$).
Since $P^i(c)\nin K$ then it is clear that for every $i$ the
family of $K$-pullbacks by $P^k$ consists of $2^{ki}$ pairwise disjoint
continua contained in $J$. Since $P^{-k}(K)\cap K\ne 0$, there exists a
$K$-pullback $K(-1)$ of order $1$ by $P^k$ non-disjoint from $K$. Such pullback
is unique. Indeed, otherwise there exists another $K$-pullback $K'$ of order
$1$ by $P^k$ which is not disjoint from $K$. Set $K(-1)\cup K\cup K'=Y$. Then
$P^k|_{Y}$ is not 1-to-1 which implies by \cite{h96} that $P^k|_Y$ has a
critical point $y\in Y$. This means that for some $j, 0\le j\le k-1$ we have
$P^j(y)=c$. If $y\in K(-1)\cup K'$ then this implies that $P^{2k-j}(c)\in P^k(K)$, a
contradiction to the assumptions. On the other hand, if $y\in K$
then $P^{k-j}(c)\in P^k(K)$, again a contradiction to the assumptions.
Hence such point $y\in Y$ does not exist, and $K(-1)$ is unique. Since
$K(-1)$ clearly satisfies the same assumptions as $K$ itself we see that the
desired sequence of pullbacks of $K$ exists and is unique.

Suppose that $P^{2k}(K)\cap K=\0$. Then $P^k(K)\cap K(-1)=\0$ and for any $i\ge
0$ we have $K(-i)\cap K(-i-2)=\0$ (otherwise we apply the appropriate power of
$P^k$ to get a contradiction). Let us show that then $K(-j)\cap P^{2k}(K)=\0$
for any $j\ge 0$. Indeed, otherwise choose the minimal such $j$ that $K(-j)\cap
P^{2k}(K)\ne \0$. Then $j\ge 1$ and by the choice of $j$ the only non-empty
intersections among sets $P^{2k}(K), P^k(K), K, \dots, K(-j)$ are intersections
among consecutive pullbacks and the intersection $K(-j)\cap P^{2k}(K)\ne \0$.
Set $P^{2k}(K) \cup P^k(K)\cup \dots \cup K(-j+1)=E$. Then

$$E\cap K(-j)=[P^{2k}(K)\cap K(-j)]\cup [K(-j+1)\cap K(-j)]$$

\noindent is disconnected as the union of two disjoint non-empty continua.
However all continua are contained in a non-separating continuum $J$ with empty
interior. Hence the intersection of any two sub-continua of $J$ must be
connected. It follows that $K(-j)\cap P^k(K)=\0$ for any $j\ge 1$ as desired.
\end{proof}

To prove Theorem~\ref{main} we need the following construction. Choose two
$P_\ta$-periodic points $u$ and $v$ of periods $k$ and $l$ respectively. Then
depending on their itineraries the strings $S_u, S_v$ of $\Da'_\ta$-pullbacks
will have a few common pullbacks and then, starting at the last common
pullback, will consist of two pairwise disjoint sequences of pullbacks.
Clearly, the points $u$ and $v$ can be chosen so that the strings $S_u$ and
$S_v$ have at least two common $\Da'_\ta$-pullbacks. For example, if $u=x$ has
itinerary $\{011\}_\iy$ and $v=z$ has itinerary $\{0110111\}_\iy$ then the
strings $S_x$ and $S_z$ have two pullbacks $\Da'_\ta=\Da_{\ta, 01_\iy}$ and
$\Da_{\ta, 01101_\iy}$ in common, yet from the third pullback on the strings
$S_x$ and $S_z$ are disjoint. In any case, and this is important for what
follows, the last common pullback of the strings $S_x$ and $S_z$ is $\Da_{\ta,
01101_\iy}$, and it is disjoint from $\Da_\ta$.

Let us assume that the basic length of $S_u$ is $w$, the basic length of $S_v$
is $q$, and the initial finite string $F$, common to both $S_u$ and $S_v$,
consists of $m\ge 2$ $\Da'_\ta$-pullbacks. For simplicity and without loss of
generality assume that $m<\min(w, q)$. Denote the last common
$\Da'_\ta$-pullback in $F$ by $L$. Consider $w$ $\Da'_\ta$-pullbacks in $S_u$
immediately following $F$ and denote their union by $\hat F_u$; also, denote
the union of $w$ $\Da'_\ta$-pullbacks in $S_u$ immediately following $\hat F_u$
by $F_u$. Similarly, consider $q$ $\Da'_\ta$-pullbacks in $S_v$ immediately
following $F$ and denote their union by $\hat F_v$; also, denote the union of
$q$ $\Da'_\ta$-pullbacks in $S_v$ immediately following $\hat F_v$ by $F_v$.
By Lemma~\ref{pullba} there exists a string of $F_u$-pullbacks by $P^k_\ta$ $F_u(0)=F_u,
F_u(-1), \dots$ such that $S_u=F\cup \hat F_u\cup (\cup^\iy_{i=0} F_u(-i))$
and a string of $F_v$-pullbacks by $P^l_\ta$ $F_v(0)=F_v,
F_v(-1), \dots$ such that $S_v=F\cup \hat F_v\cup (\cup^\iy_{i=0} F_v(-i))$.

Set $n=m+3k+3l$. We choose $\al\in \tS_N$ very close to $\ta$ so that a few
conditions are satisfied. By Theorem~\ref{bc} we may assume that the area of
$J(P_\al)$ is positive and that $P_\al$ is a basic Cremer polynomial. We use
Lemma~\ref{inter} to guarantee that $\partial [M_{\al}, \Da_\ta]$ is so small
that the first $n$ iterates of $c_\al$ avoid $M'_\al$ (so that
$M'_\al$-pullbacks by $P_\al$ of order at most $n$ map to $M'_\al$ univalently), all
$M'_\al$-pullbacks by $P_\al$ of order $n$ are very close to the corresponding
$\Da'_\ta$-pullbacks by $P_\ta$ and the intersection patterns of $\Da'_{\ta,
\ol{i}}, \ol{i}\in A_n$ and $M'_{\al, \ol{i}}, \ol{i}\in A_n$ are the same.
Hence we may talk about the strings $F^\al, \hat F^\al_u, \hat F^\al_v,
F^\al_u, F^\al_v$ of $M'_\al$-pullbacks by $P_\al$ which correspond to the
strings $F, \hat F_u, \hat F_v, F_u, F_v$, and all the sets $F^\al, \hat
F^\al_u, \hat F^\al_v, F^\al_u, F^\al_v$ are continua. Denote by $L^\al$ the
$M'_\al$-pullback corresponding to the $\Da'_\ta$-pullback $L$. Recall that
since $m\ge 2$ then $L\cap \Da_\ta=\0$ and hence (by the choice of $\al$) we
may assume that $M_\al\cap (L^\al\cup \hat F^\al_u\cup \hat F^\al_v)=\0$.

Consider the set $F^\al_u$. Then $P^k_\al(F^\al_u)=\hat F^\al_u$ is disjoint
from $M_\al=\om(c_\al)$ and hence $P_\al^i(c_\al)\nin P^k_\al(F^\al_u)$ for
any $i\ge 0$. On the other hand, by the choice of $n$ we have that
$P^{-k}(F^\al_u)\cap F^\al_u\ne \0$.
%(the pullback of $F^\al_u$ non-disjoint
%from $F^\al_u$ corresponds to the $\Da'_\ta$-pullbacks in $S_u$ which can be
%numbered $m+2k+1, \dots, m+3k$).
Hence by Lemma~\ref{pullba} there is a
sequence of $F^\al_u$-pullbacks $F^\al_u=F^\al_u(0), F^\al_u(-1), \dots$ by $P^k_\al$.
Observe that $F^\al_u$ is disjoint from $P^{2k}(F^\al_u)$ by the choice of
$\al$ (clearly, $P^{2k}(F^\al_u)=P^k_\al(\hat F^\al_u)$ is the string of
$M'_\al$-pullbacks by $P^k_\al$ connecting $L^\al$ and $M_\al$ united with
$M_\al$ itself). By Lemma~\ref{pullba} we conclude that the set

$$\hat F^\al_u\cup \bigcup^\iy_{i=0} F^\al_u(-i)=Q^\al_u$$
is a chain of ``concatenated'' continua such that intersections among them are
only possible between $\hat F^\al_u$ and $F^\al_u=F^\al_u(0)$ and between two
consecutive pullbacks of $F^\al_u$.

Analogous claims can be proven for $\hat F^\al_v$ and $F^\al_v$. For the
corresponding pullbacks of $F^\al_v$ we use similar notation and get the set

$$\hat F^\al_v\cup \bigcup^\iy_{i=0} F^\al_v(-i)=Q^\al_v,$$
a chain of ``concatenated'' continua such that intersections among them are
only possible between $\hat F^\al_v$ and $F^\al_v=F^\al_v(0)$ and between two
consecutive pullbacks of $F^\al_v$.

It is easy to see that the sets $Q^\al_u$ and $Q^\al_v$ are disjoint. Indeed,
suppose otherwise. Then we can choose minimal $r, s$ such that $F^\al_u(-r)\cap
F^\al_v(-s)\ne \0$. It follows that $r>1$ and $s>1$. Then the continua $X=L^\al
\cup (\cup^r_{i=0} F^\al_u(-i))$ and $Y=L^\al \cup (\cup^s_{i=0}F^\al_v(-i))$
have a disconnected intersection (it consists of $L^\al$ and a non-empty
compact set $F^\al_u(-r)\cap F^\al_v(-l)$ disjoint from $L^\al$) despite the
fact that they both are contained in the Julia set $J(P_\al)$ (recall that the
Julia sets of basic Cremer polynomials are non-separating continua with empty
interior).

Since $M_\al\cap (\hat F^\al_u\cup \hat F^\al_v)=\0$ by the choice of $\al$
then by the Shrinking Lemma~\cite{lyumin} we know that $\diam(F^\al_u(-i))\to
0$ as $i\to \iy$ (i.e., $F^\al_u(-i), i=0, 1, \dots$ is a null-sequence). Hence
by continuity any limit point $a$ of this sequence of sets is $P^k_\al$-fixed.
Since there are finitely many $P^k_\al$-fixed points while the set of all limit
points of the sequence of sets $F^\al_u(-i), i=0, 1, \dots$ is connected
(recall, that this sequence of sets is a chain of ``concatenated'' continua) we
conclude that the sequence $F^\al_u(-i), i=0, 1, \dots$ converges to a
$P^k_\al$-fixed point which we will denote by $u'$. Similarly, the sequence
$F^\al_v(-i), i=0, 1, \dots$ converges to a $P^l_\al$-fixed point which we will
denote by $v'$.

Let us study possible intersections between some of these sets. Set
$Z=P^k_\al(\hat F^\al_u)=P^l_\al(\hat F^\al_v)$ ($Z$ is the string
of pullbacks of $M'_\al$ connecting $L^\al$ and $M_\al$, united with $M_\al$).
Lemma~\ref{pullba} implies that

$$[Z\cup \hat F^\al_u]\cap Q^\al_u=\hat F^\al_u$$

and

$$[Z\cup \hat F^\al_v]\cap Q^\al_v=\hat F^\al_v.$$

Let us show that $u'\nin Z\cup Q^\al_u$. Indeed, since $u'$ is a $P^k$-fixed
point then $u'\nin Q^\al_u$ because of the way sets $F^\al_u(-i), i\ge 0$
intersect. Suppose that $u'\in Z$ and consider two continua, $X=\ol{Q^\al_u}$
and $Y=Z\cup \hat F^\al_u$. It follows that their intersection $X\cap
Y=\{u'\}\cup \hat F^\al_u$ is disconnected, a contradiction. Thus, $u'\nin
Z\cup Q^\al_u$. Similarly we can show that $v'\nin Z\cup Q^\al_v$. Analogous
arguments show that since the continua $\ol{Z\cup Q^\al_u}$ and $\ol{Z\cup
Q^\al_v}$ must have a connected intersection then $u'\nin Q^\al_v, v'\nin
Q^\al_v$ and $u'\ne v'$. Finally, since by the construction $M_\al\cap (L^\al
\cup\hat F^\al_u\cup \hat F^\al_v)=\0$ then $M_\al\cap \ol{L^\al\cup
Q^\al_u\cup Q^\al_v}=\0$ (notice that by the above $u', v'\nin M_\al$.

Let $R_\ga$ be the external ray for $P_\al$ landing at $u'$ and $R_{\be}$ be
the external ray for $P_\al$ landing at $v'$. The above implies that the union

$$\ol{L_\al \cup Q^\al_u \cup
Q^\al_v} \cup R_\ga \cup R_\be$$

\noindent cuts the plane into two open half-planes $W$ and $H$, one of which
(say, $W$) contains $M_\al$ (and therefore the Cremer point $p_\al$ of
$P_\al$). Choose any external angle $\tau$ whose external ray $R_\tau$ is
contained in $H$. Then it follows that the impression of $R_\tau$ is contained
in $\ol{L_\al \cup Q^\al_u \cup Q^\al_v}$ and hence does not contain $p_\al$.
By Theorem~\ref{mainold} this implies that $J(P_\al)$ is a solar Julia set which completes
the proof of Theorem~\ref{main}.

\bibliographystyle{amsalpha}
\bibliography{/lex/references/refshort}
\providecommand{\bysame}{\leavevmode\hbox to3em{\hrulefill}\thinspace}
\providecommand{\MR}{\relax\ifhmode\unskip\space\fi MR }
\providecommand{\MRhref}[2]{%
\href{http://www.ams.org/mathscinet-getitem?mr=#1}{#2}
} \providecommand{\href}[2]{#2}

\end{document}